\documentclass[fontsize=12pt,paper=letter, bibliography=totoc]{scrartcl}
\usepackage[hmargin=1in, vmargin=1in]{geometry}

\usepackage{graphicx}
\usepackage{verbatim}
\usepackage[sumlimits,]{amsmath}
\usepackage{amsfonts}
\usepackage{amsthm}
\usepackage{amssymb}
\usepackage{mathrsfs}
\usepackage{srcltx}
\usepackage{bbm}

\usepackage{graphicx,epsfig}
\usepackage{enumitem} 
\usepackage{hyperref}

\newcommand{\calB}{\mathcal{B}}

\newcommand{\calF}{\mathcal{F}}

\newcommand{\calO}{\mathcal{O}}

\newcommand{\calS}{\mathcal{S}}
\newcommand{\calU}{\mathcal{U}}
\newcommand{\calM}{\mathcal{M}}
\newcommand{\calN}{\mathcal{N}}
\newcommand{\calX}{\mathcal{X}}

\newcommand{\dt}{{\,\mathrm{dt}}}

\newcommand{\dx}{{\, \mathrm{dx}}}
\newcommand{\dy}{{\, \mathrm{dy}}}

\newcommand{\deta}{{\,\mathrm{d}\eta}}
\newcommand{\dmu}{{\,\mathrm{d}\mu}}
\newcommand{\dnu}{{\,\mathrm{d}\nu}}

\newcommand{\iid}{{\mbox{i.i.d.}\,}}
\newcommand{\rmP}{\mathrm{P}}

\newcommand{\st}{{:\,}}

\newcommand{\as}{\mbox{a.s.}}

\newcommand{\bfZ}{\mathbf{Z}}
\newcommand{\eps}{\varepsilon}

\newcommand{\one}{\mathbf{1}}

\newcommand{\Cov}[1]{\mbox{Cov}\left(#1\right)}
\newcommand{\Exp}[1]{\E\left\{ #1 \right\}}
\newcommand{\Exparg}[2]{\E_{#1}\left\{ #2 \right\}}

\newcommand{\innerproduct}[2]{{\langle#1,#2\rangle}}
\newcommand{\Var}[1]{\mbox{Var}\left\{ #1 \right\}}

\newcommand{\E}{{\mathbb{E}}}
\newcommand{\g}{\,|\,}

\newcommand{\qbox}[1]{\quad\mbox{#1}\quad}

\newcommand{\dsphere}{{\mathbb{S}^{d-1}}}

\newcommand{\Nat}{\mathbb{N}}

\newcommand{\Rd}{{{\mathbb{R}}^d}}

\newcommand{\Rone}{\mathbb{R}}

\newcommand{\dabseta}{{\,\mathrm{d}|\eta|}}
\newcommand{\dof}{{f}}
\newcommand{\D}{{\mathcal{D}^2}}
\newcommand{\Dtilde}{{\tilde{\mathcal{D}}^2}}

\newcommand{\domega}{{\,\mathrm{d}\omega}}
\newcommand{\etahat}{{\hat{\eta}}}
\newcommand{\muhat}{{\hat{\mu}}}
\newcommand{\nuhat}{{\hat{\nu}}}

\newcommand{\F}{{\calF_0}}

\newcommand{\Vol}{{\mathrm{Vol}}}

\newcommand{\Xbar}{{\bar{X}}}
\newcommand{\Ybar}{{\bar{Y}}}
\newcommand{\support}{{\calS_{U,\calB}}}
\newcommand{\supportspan}{{\mathrm{Span}({\support})}}

\newcommand{\supportspanclosure}{{\overline{\supportspan}}}
\newcommand{\spacenoindent}{{\vspace{1cm}\noindent}}

\newtheorem{theorem}{Theorem}[section]
\newtheorem*{theorem*}{Theorem}
\newtheorem{proposition}[theorem]{Proposition}
\newtheorem{lemma}[theorem]{Lemma} 
\newtheorem{corollary}[theorem]{Corollary}

\theoremstyle{definition}
\newtheorem*{definition*}{Definition}
\newtheorem{definition}{Definition}[section]

\theoremstyle{remark}
\newtheorem*{remark*}{Remark}

\newtheorem*{claim*}{Claim}

\begin{document}
\title{
    Random Field Representations of Kernel Distances
}
\author{
    Ian Langmore\\
    \small Gridmatic Inc.
    \small \texttt{ianlangmore@gmail.com}
}
\maketitle
\abstract{
    Positive semi-definite kernels are used to induce pseudo-metrics, or ``distances'', between measures. We write these as an expected quadratic variation of, or expected inner product between, a random field and the difference of measures. This alternate viewpoint offers important intuition and interesting connections to existing forms. Metric distances leading to convenient finite sample estimates are shown to be induced by fields with dense support, stationary increments, and scale invariance. The main example of this is energy distance. We show that the common generalization preserving continuity is induced by fractional Brownian motion. We induce an alternate generalization with the Gaussian free field, formally extending the Cram\'er-von Mises distance. Pathwise properties give intuition about practical aspects of each. This is demonstrated through signal to noise ratio studies.
}

\spacenoindent
\tableofcontents

\medskip

\section{Introduction}
\label{section:introduction}
Let $\mu$ and $\nu$ be probability measures on Borel subsets of $\calX\subset\Rd$.
Given positive semi-definite kernel $k:\calX\times\calX\to\Rone$, a pseudo-metric, or (squared) distance, can be defined
\begin{align*}
  \D(\mu - \nu) &= \Exp{k(X, X')} + \Exp{k(Y, Y')} - 2 \Exp{k(X, Y)},
\end{align*}
where $X, X'\sim\mu$, $Y, Y'\sim\nu$, and all four are independent.

This can also be represented as a weighted squared $L^2$ distance between characteristic functions, as a maximum mean discrepancy (MMD)~\cite{Gretton2012-fp}, or the divergence form of a scoring rule~\cite{Gneiting2007-ms}.

We explore an alternative viewpoint where $\D$ is induced by a random field $U$, defined on $\calX$.
\begin{align*}
  \D(\mu - \nu) &= \Exparg{U}{\Exparg{X\sim\mu,Y\sim\nu}{U(X) - U(Y)}^2\g U} \\
  &= \Exparg{U}{\innerproduct{U}{\mu - \nu}^2}.
\end{align*}

We could find no literature expressing this viewpoint, which is somewhat surprising. The closest is the notion of \emph{Brownian distance covariance}~\cite{Szekely2009-ho}.
We hope to convince the reader that this viewpoint provides important intuitions.

As notation suggests, the distance depends only on the difference $\mu-\nu$. Therefore, our assumptions will usually be on spaces of finite signed Borel measures.
These have a unique representation as differences of finite Borel measures.
\begin{align*}
    \F(\calX) :&= \left\{ \eta = \mu - \nu\st \mu, \nu\mbox{ are finite Borel measures on }\calX, \eta(\calX)=0 \right\}.
\end{align*}
This is larger than differences of \emph{probability} measures, but is needed for our support theorems where we want (subsets of ) $\F$ to be a vector space.

The first example of this is also the most important.
\begin{definition}
    We say field $U$ is characteristic over $\calM\subset\F(\calX)$ if, for all $\mu-\nu\in\calM$, $\D(\mu -\nu)=0$ only when $\mu=\nu$.
    \label{def:characteristic}
\end{definition}
This terminology is borrowed from~\cite{Sriperumbudur2009-ia}, where the \emph{kernel} is called characteristic if $\D$ is a metric.

We proceed in section~\ref{section:equivalence-of-representations} by showing equivalence of different representations. Section~\ref{section:crps-energy-fractional-brownian-motion} analyzes energy distance as induced by continuous fractional Brownian fields, and introduces the \emph{Dirichlet energy distance} as a generalization induced by the Gaussian free field. It concludes with a practical application to signal to noise ratio.
Section~\ref{section:other-examples} discusses paths on discrete spaces and additive paths.
Section~\ref{section:natural-assumptions-for-the-field} shows distances will be most useful if sample paths satisfy three requirements. Specifically, if $\mu-\nu$ lives in Banach dual $\calB^\ast$, then sample paths should have wide sense stationary increments, a scale invariance or fractal form, and live in $\calB$.
This places restrictions on available fields, a subject which we discuss throughout.
Finally, section~\ref{section:forgery-theorems} shows that the support of characteristic fields have dense span in $\calB$, and if the fields are Gaussian their support is dense. This allows the support of a number of fields to be characterized.

\section{Different representations and their equivalence}
\label{section:equivalence-of-representations}
Let $U:\Omega\times\calX\to\Rone$ be a random field.  For $\zeta\in\Omega$, sample paths $U(\zeta, \cdot)$ can be generalized functions.
As usual, we suppress reference to $\Omega$ or $\zeta$ when describing random fields.

If $\mu$ is a probability measure, the characteristic function is $\Exparg{X\sim\mu}{e^{i\omega\cdot X}}$. Since we work with finite signed measures, and use the word ``characteristic'' for other meanings, we instead write
\begin{align*}
    \etahat(\omega) :&= \int_\calX e^{i\omega\cdot x}\deta(x)
\end{align*}
and call it the Fourier transform of the measure $\eta$. Since, for $\eta\in\F(\calX)$, $|\eta|(\calX)<\infty$, $\etahat$ is a bounded and uniformly continuous function.

\subsection{Quadratic variation of random field}

Define the squared distance $\D$ as,
\begin{align}
  \D(\mu-\nu) &= \Exparg{U}{\Exparg{X\sim\mu,Y\sim\nu}{U(X) - U(Y)\g U}^2}
  \label{align:distance-stochastic-integral}
\end{align}
We can re-write~\eqref{align:distance-stochastic-integral} to emphasize an inner product between sample paths $U$ and the difference of measures
\begin{align}
  \D(\mu-\nu) &= \Exparg{U}{\innerproduct{U}{\mu-\nu}^2}
  \label{align:distance-stochastic-pairing}
\end{align}
where the pairing can also be written
\begin{align*}
    \innerproduct{U}{\mu - \nu} &:= \int_\calX U(x) (\dmu(x) - \dnu(x)).
\end{align*}

The distance will be highly sensitive to differences $\mu-\nu$ that look similar to ``typical'' paths.
For example, if paths look (almost) linear, then $\D$ will (mostly) measure differences in the first moment.

\subsection{Kernel representation}
By construction, $\Cov{U(x)U(x')}=k(x,x')\in\Rone$ is positive semi-definite. This means, for $n\in\Nat$, $\beta_1,\ldots,\beta_n\in\Rone$, and $x_1,\ldots,x_n\in\calX$,
\begin{align*}
    \sum_{i,j=1}^n \beta_i\beta_jk(x_i,x_j) \geq0.
\end{align*}
In this case, the kernel~\eqref{align:kernel-distance-traditional} and stochastic integral~\eqref{align:distance-stochastic-integral} representations are equivalent.

We show this following steps, which are very similar to the construction of Brownian distance covariance in~\cite{Szekely2009-ho}. Proceeding, with $X,X'\sim\mu$, and $Y,Y'\sim\nu$ (all independent of each other),
\begin{align*}
  &\Exparg{U}{\Exparg{X,Y}{U(X) - U(Y)\g U}^2} \\
  &\quad=\Exparg{U}{\Exparg{X,Y}{U(X) - U(Y)\g U}\Exparg{X',Y'}{U(X')-U(Y')\g U}} \\
&\quad=\Exparg{U}{\Exparg{X,X',Y,Y'}{U(X)U(X') + U(Y)U(Y') - U(X)U(Y') - U(Y)U(X')\g U}} \\
&\quad=\Exparg{X,X',Y,Y'}{\Exparg{U}{U(X)U(X') + U(Y)U(Y') - U(X)U(Y') - U(Y)U(X')\g X,X',Y,Y'}} \\
&\quad=\Exparg{X,X',Y,Y'}{k(X, X') + k(Y, Y') - k(X, Y') - k(Y, X')}.
\end{align*}

Using the symmetry of the kernel, we arrive at the kernel representation
\begin{align}
  \D(\mu-\nu) &= \Exp{k(X, X')} + \Exp{k(Y, Y')} - 2 \Exp{k(X, Y)},
  \label{align:kernel-distance-traditional}
\end{align}
as promised.

\subsection{Fourier representation induced by Wiener measure}
\label{section:fourier-representations}
Here we use a non-negative spectral density $\varphi$ to construct $U$ on $\Rd$, that will be stationary or have stationary increments.
Inserting the pathwise definition of $U$ into~\eqref{align:distance-stochastic-pairing} will lead to an equivalent Fourier~\eqref{align:spectral-representation-of-distance} representation.

Let us start by briefly reviewing Wiener integrals.
Using definitions from~\cite{Cohen2013-zp} 2.1.6.1,
$W(\dx)$ is a real Wiener measure on $\Rd$ if, for $f,g\in L^2(\Rd)$,
\begin{align*}
  \Exp{\left( \int f(x)W(\dx) \right)\left( \int g(x)W(\dx) \right)}
  &= \frac{1}{(2\pi)^{d/2}}\int f(x)g(x)\dx.
\end{align*}
Letting $\hat{f}$ denote the Fourier transform of $f$, the Fourier transform of the Wiener measure $\widehat{W}$ is defined by.
\begin{align*}
  \int f(\omega)\widehat{W}(\domega)
  &= \int \hat{f}(x)W(\dx),
\end{align*}
so that Parseval's identity gives
\begin{align}
  \label{align:brownian-measure-fourier-transform-parseval}
  \Exp{\left( \int f(x)\widehat{W}(\dx) \right)\left( \int g(x)\widehat{W}(\dx) \right)}
  &= \frac{1}{(2\pi)^{d/2}}\int f(x)\overline{g}(x)\dx.
\end{align}

We can construct stationary $U$ by following~\cite{Cohen2013-zp} section 2.1.11 to pair with $\calM\subset\F(\Rd)$ as follows. First, suppose spectral density $\varphi$ satisfies
\begin{align}
    \label{align:spectral-density-assumptions-stationary-fields}
    \begin{split}
        &\omega\mapsto\sqrt{\varphi(\omega)} \in L^2, \\
        &\omega\mapsto\etahat(\omega)\sqrt{\varphi(\omega)} \in L^2
        \qbox{for every}\eta\in\calM.
    \end{split}
\end{align}
Then
\begin{align*}
  U(x) &= (2\pi)^{d/4}\int e^{i\omega\cdot x}\sqrt{\varphi(\omega)}\widehat{W}(\domega)
\end{align*}
and the pairing
\begin{align}
  \label{align:spectral-pairing-of-stationary-fields}
    \innerproduct{U}{\eta}
    &=
    (2\pi)^{d/4}\int \etahat(\omega)\sqrt{\varphi(\omega)}\widehat{W}(\domega),
\end{align}
are well defined.
Furthermore, using~\eqref{align:brownian-measure-fourier-transform-parseval} we have continuous and bounded kernel
\begin{align*}
  k(X, Y)
  &= \Exp{U(X)U(Y)}
  = \int e^{i\omega\cdot (X - Y)}\varphi(\omega)\domega.
\end{align*}

Next we construct $U$ with stationary increments, meaning $U(x) - U(y)\sim U(x + \delta) - U(y + \delta)$, for all $x, y, \delta\in\Rd$, again pairing with $\calM\subset\F(\Rd)$.
Assume spectral density $\varphi$ satisfies
\begin{align}
    \label{align:spectral-density-assumptions-stationary-increment-fields}
    \begin{split}
        \omega&\mapsto \left( e^{i\omega\cdot x} - 1 \right) \sqrt{\varphi(\omega)}\in L^2,
        \qbox{for every} x\in\Rd,\\
        \omega&\mapsto \etahat(\omega) \sqrt{\varphi(\omega)} \in L^2,
        \qbox{for every} \eta\in\calM.
    \end{split}
\end{align}
Then
\begin{align}
    U(x) &= (2\pi)^{d/4}\int (e^{i\omega\cdot x} - 1)\sqrt{\varphi(\omega)}\widehat{W}(\domega),
    \label{align:spectral-representation-of-stationary-increment-fields}
\end{align}
and using $\int 1\deta(x)=\eta(\Rd)=0$, the pairing once again is
\begin{align}
  \label{align:spectral-pairing-of-stationary-increment-fields}
    \innerproduct{U}{\eta} &= (2\pi)^{d/4} \int
    \etahat(\omega)
    \sqrt{\varphi(\omega)}\widehat{W}(\domega),
\end{align}
The factors $(e^{i\omega\cdot x} - 1)$ and $\etahat(\omega)$ both vanish at the origin, which allows $\varphi$ to be unbounded. This is necessary for continuous non-stationary fields that have unbounded covariance at infinity.
The covariance kernel this time is
\begin{align*}
  k(X, Y)
  &= \int \left( e^{i\omega\cdot X} - 1 \right)\left( e^{-i\omega\cdot Y} - 1 \right)\varphi(\omega)\domega.
\end{align*}

Plugging either~\eqref{align:spectral-pairing-of-stationary-fields} or~\eqref{align:spectral-pairing-of-stationary-increment-fields} into~\eqref{align:distance-stochastic-integral}, and using~\eqref{align:brownian-measure-fourier-transform-parseval}, we have the Fourier representation of the distance
\begin{align}
  \begin{split}
    \D(\mu-\nu)
    &= \E_{W,W'}\left\{
      (2\pi)^{d/2}
  \left(\int\left(\muhat(\omega)-\nuhat(\omega)\right)\sqrt{\varphi(\omega)}\widehat{W}(\domega)\right) \right. \\
  &\qquad\qquad\left.\times\left(\int\left(\muhat(\omega)-\nuhat(\omega)\right)\sqrt{\varphi(\omega)}\widehat{W'}(\domega)\right)
  \right\} \\
  &= \int \left|\muhat(\omega)-\nuhat(\omega)\right|^2 \varphi(\omega)\domega.
  \end{split}
  \label{align:spectral-representation-of-distance}
\end{align}

\eqref{align:spectral-representation-of-distance} provides a well known way to check if $\D$ is characteristic.
\begin{proposition}
  \label{proposition:spectral-criteria-for-characteristic}
  Let $\calM\subset\F(\Rd)$.
  Suppose Gaussian field $U$ is stationary with spectral density satisfying~\eqref{align:spectral-density-assumptions-stationary-fields}, or has stationary increments with density satisfying~\eqref{align:spectral-density-assumptions-stationary-increment-fields}. Then $U$ is characteristic over $\calM$ if and only if $\varphi > 0$ on the support of $\calM$.
\end{proposition}

The stationary spectrum requirements are easily met if $\varphi\in L^1$.
We also have an easier way to check the non-stationary requirements~\eqref{align:spectral-density-assumptions-stationary-increment-fields}.
\begin{lemma}
    Suppose there exists $\kappa\in[0, 1]$, $C_\kappa<\infty$ such that
    \begin{align*}
        &\int_{\|\omega\|\leq1} |\varphi(\omega)| \|\omega\|^{2\kappa}\domega < \infty,
        \qquad
        \int_{\|\omega\|> 1}\varphi(\omega)\domega < \infty, \\
        &\quad\int \|x\|^\kappa d|\eta|(x) < C_\kappa.
    \end{align*}
    Then,~\eqref{align:spectral-density-assumptions-stationary-increment-fields} are met.
    \label{lemma:meeting-spectral-density-assumptions-with-moments}
\end{lemma}
\begin{proof}
    Since, for $p\in[0, 1]$, $\min(a, b) \leq a^p b^{1-p}$,
    we may use~\cite{Durrett2019-fw} lemma 3.3.19 to see
    \begin{align*}
        \left| e^{i\omega\cdot x} - 1 \right|
        &\leq \min\left( |\omega\cdot x|, 2 \right)
        \leq 2^{1-p}|\omega\cdot x|^p
        \leq 2^{1-p}\|\omega\|^p\|x\|^p.
    \end{align*}
    we have
    \begin{align*}
        |\etahat(\omega)|
        &=\left|
        \etahat(\omega) - \eta(\Rd)
        \right|
        =\left|
        \int \left( e^{i\omega\cdot x} - 1 \right)\deta(x)
        \right|
        \leq
        2^{1-\kappa}
        \|\omega\|^\kappa C_\kappa,
    \end{align*}
    and of course $|\etahat(\omega)|\leq 2 |\eta|(\Rd)$. Putting these together
    \begin{align*}
        &\int_{\|\omega\| \leq 1}\left| \etahat(\omega) \right|^2\varphi(\omega)\domega
        + \int_{\|\omega\| > 1}\left| \etahat(\omega) \right|^2\varphi(\omega)\domega \\
        &\qquad\leq
        2^{2(1-\kappa)}C_\kappa^2 \int_{\|\omega\|\leq 1}\|\omega\|^{2\kappa}\varphi(\omega)\domega
    + 2^2(|\eta|(\Rd))^2\int_{\|\omega\|> 1}\varphi(\omega)\domega < \infty.
    \end{align*}
    This proves the second condition of~\eqref{align:spectral-density-assumptions-stationary-increment-fields}. The first is similar.
\end{proof}

\section{Brownian distance and its generalizations}
\label{section:crps-energy-fractional-brownian-motion}

\subsection{One dimensional Brownian distance}
For $-\infty < t < \infty$, let $B(t)$ be a two-sided Brownian motion. This is constructed by patching together two independent one sided motions. That is  $B(t) = B_1(t)$ if $t\geq0$, or $B_2(-t)$ if $t<0$. One may also simply define $B$ as a Gaussian random field with
\begin{align*}
  B(0) &= 0,
  \quad
  \E B(t)\equiv 0,
  \qbox{and}
  \E B(t)B(s) = \min(t,s) = (|t| + |s| - |t - s|) / 2.
\end{align*}

The corresponding kernel distance is found using~\eqref{align:kernel-distance-traditional},
\begin{align}
  \D(\mu-\nu; B) &= \E|X - Y| - (1/2)\,\E|X - X'| - (1/2)\,\E|Y - Y'|,
  \label{align:energy-distance-1d}
\end{align}
with $X,X'\sim\mu$, $Y, Y'\sim\nu$.
Integrating by parts and using It\^{o}'s isometry, we have equivalence to the Cram\'er-von Mises distance.
\begin{align}
  \begin{split}
    \D(\mu-\nu; B) &=
    \Exparg{B}{
      \left( 
        \int_{-\infty}^\infty B(t) (\dmu(t) - \dnu(t))
      \right)^2
    } \\
    &=\Exparg{B}{
      \left( 
        \int_{-\infty}^\infty (\mu(-\infty, t) - \nu(-\infty, t))dB(t)
      \right)^2
    } \\
    &= \int_{-\infty}^\infty (\mu(-\infty, t) - \nu(-\infty, t))^2\dt.
  \end{split}
  \label{align:energy-distance-1d-ito-isometry}
\end{align}

\subsection{Extending with continuous fractional fields}
\label{section:extending-with-continuous-fractional-fields}
To extend Brownian distance while preserving the form of the kernel representation~\eqref{align:energy-distance-1d}, we replace Brownian motion with continuous fractional fields.

Following~\cite{Cohen2013-zp} section 3.3.1, for \emph{Hurst index} $H\in(0, 1)$
we define $B^H$ to be the Gaussian field with stationary increments on $\Rd$, $B^H(0)=0$, $\E B^H\equiv0$, and covariance
\begin{align*}
  \Exp{B^H(x)B^H(y)} &= \frac{1}{2}\left( \|x\|^{2H} + \|y\|^{2H} - \|x - y\|^{2H} \right).
\end{align*}
The spectral density in the sense of~\eqref{align:spectral-representation-of-stationary-increment-fields} is
\begin{align*}
  \varphi(\omega) &= \frac{1}{(2\pi)^{d/2}C_H}\frac{1}{\|\omega\|^{d + 2H}},
  \qbox{where}
  C_H := \frac{\pi^{1/2}\Gamma(H + 1/2)}{2^{d/2} H \Gamma(2H) \sin(\pi H)\Gamma(H + d/2)}.
\end{align*}
This results in kernel and Fourier representations
\begin{align}
  \begin{split}
    \D(\mu-\nu; B^H)
    &= \E \|X - Y\|^{2H} - (1/2)\E\|X - X'\|^{2H} - (1/2)\E\|Y - Y'\|^{2H}, \\
    &= \frac{1}{(2\pi)^{d/2}C_H}\int\frac{\left| \muhat(\omega)-\nuhat(\omega) \right|^2}{\|\omega\|^{d + 2H}}\domega.
  \end{split}
\end{align}
Using proposition~\ref{proposition:spectral-criteria-for-characteristic} and lemma~\ref{lemma:meeting-spectral-density-assumptions-with-moments} we have
\begin{corollary}
  \label{corollary:continuous-fractional-fields-are-characteristic}
    For $H<\kappa\leq1$, the continuous fractional fields $B^H$, $H\in(0, 1)$ are characteristic over
    \begin{align*}
        \calM :&= \left\{ \eta\in\F(\Rd)\st\int\|x\|^\kappa\dabseta(x)<\infty \right\}.
    \end{align*}
\end{corollary}
This result is well-known for the kernel representation~(\cite{Szekely2013-qh}), which requires the stricter $\int \|x\|^{2H}\dabseta(x)< \infty$.

It is apparent that along any line through the origin, $B^H$ behaves as a one dimensional fractional motion. In fact, the restriction of $B^H$ to any $k$ dimensional plane is a shifted version of a $k-$dimensional fractional field.

The increments of $B^H$ satisfy
\begin{align*}
  &\Exp{(B^H(t) - B^H(s))(B^H(t')-B^H(s'))} \\
  &\quad= \frac{1}{2} \left( |s' - t|^{2H} + |t' - s|^{2H} - |t' - t|^{2H} - |s' - s|^{2H} \right).
\end{align*}
Supposing the increments do not overlap, $s<t<s'<t'$, they are positively correlated when $H > 1/2$, negatively when $H < 1/2$, and independent when $H=1/2$~\cite{Shevchenko2015-wg}.
Further analysis shows that $B^H$ exhibits long range dependence just when $H > 1/2$~\cite{Ryvkina2013-ax}.
Traces of fractional motion are shown in figure~\ref{fig:fractional-brownian-1d-traces}. We see small Hurst index corresponds to fast decorrelation, and larger index to nearly linear behavior.
These pathwise behaviors hint that when $H$ is small (large), $\innerproduct{B^H}{\mu-\nu}^2$ will be more sensitive to even (odd) moments. This can be studied in the limits.
\begin{figure}[h] 
  \centering 
  \includegraphics[
	  width=0.9\textwidth,  
  ]{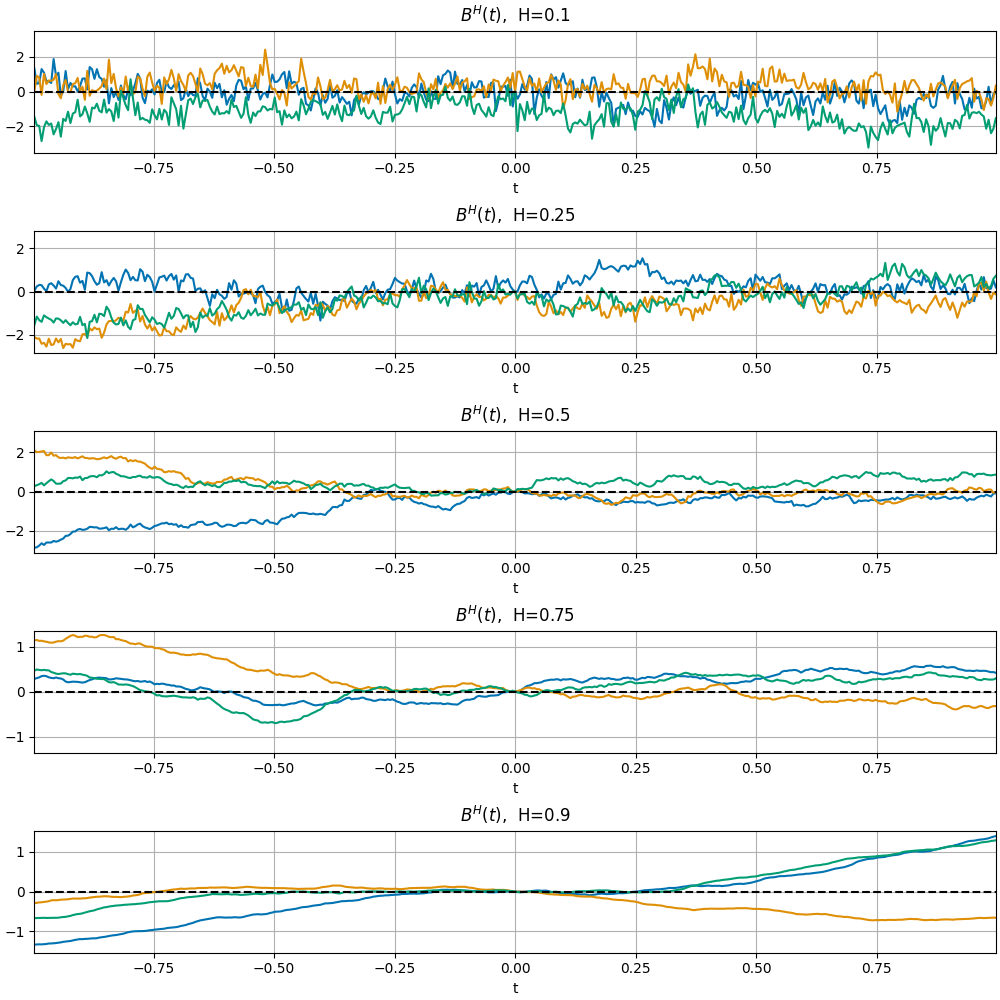} 
  \caption{\textbf{Traces of fBM:} Smaller Hurst index $H$ results in a field that decorrelates quickly, and approximates a (random) constant plus stationary noise. $H=0.5$ corresponds to Brownian motion. As $H\to1$, the traces become linear with random slope.}
    \label{fig:fractional-brownian-1d-traces}
\end{figure} 

As $H\to1$, $\Exp{B^H(x)B^H(y)}\to x\cdot y$. Therefore, on bounded sets, $B^H\to L$ in distribution, where $L$ is the linear process $L(x) \sim \bfZ\cdot x$, with $\bfZ\sim\calN(0, I_d)$.
We therefore have, as $H\to1$, for finite signed Borel measure $\eta$,
\begin{align}
  \label{align:fractional-distance-large-h-limit}
  \Exparg{B^H}{\innerproduct{B^H}{\eta}^2}&\to
  \Exparg{\bfZ}{\left( \sum_{i=1}^d \bfZ_i \int x_i\deta(x) \right)^2}
  &= \sum_{i=1}^d\left( \int x_i\deta(x) \right)^2 .
\end{align}
In particular,~\eqref{align:fractional-distance-large-h-limit} shows that $\Exparg{B^H}{\innerproduct{B^H}{\mu-\nu}^2} \to \|\Xbar-\Ybar\|^2$, as $H\to1$.

On the other hand, for fixed $x\neq y$, as $H\to0$, $\Exp{B^H(x)B^H(y)}\to 1/2$, whereas $\Exp{B^H(x)^2}\to1$. Therefore, on bounded sets, for $0 < H \ll 1/2$, we expect $B^H$ to behave like $(Z + V(x))/\sqrt{2}$, where $Z\sim\calN(0, 1)$ and $V(x)$ approximates the stationary discrete white noise process, $\E V(x)V(y)=\one_{\|x - y\|<\eps}$, for some small correlation length $\eps\to0$ as $H\to0$.
For continuous $q\in L^1\cap L^2$, as $H\to0$,
\begin{align}
  \label{align:fractional-distance-small-h-limit}
  \begin{split}
    \Exparg{B^H}{\innerproduct{B^H}{q}^2}
    &\to \Exparg{Z,V}{\left( Z\int q(x)\dx + \int q(x)V(x)\dx \right)^2}
    = \frac{1}{2}\left( \int q(x)\dx \right)^2.
  \end{split}
\end{align}
In particular,~\eqref{align:fractional-distance-small-h-limit} shows that, if $\mu$ and $\nu$ have densities $\in C(\Rd)\cap L^1(\Rd)$, $\Exparg{B^H}{\innerproduct{B^H}{\mu-\nu}^2}\to 0$ (consistent with the ``signal'' in figure~\ref{fig:snr-sweep}).

Using the limiting forms, we can gain intuition about which moments will contribute most to $\D$.
First, use~\eqref{align:fractional-distance-large-h-limit} and assume $\eta$ has density $q_m(t) = \one_{t\in[-1,1]}t^m$. Then, as $H\to1$,
\begin{align*}
  \Exparg{B^H}{\innerproduct{B^H}{q_m}^2}
  &\to
  \left\{ 
    \begin{matrix}
      4 / (m+2)^2,&\qbox{$m$ is odd,}\\
      0,&\qbox{$m$ is even.}
    \end{matrix}
\right.
\end{align*}
Second, with~\eqref{align:fractional-distance-small-h-limit} we can explicitly compute the small $H$ limit of moments of $U$ on $[-1, 1]$,
and, as $H\to0$,
\begin{align*}
  \Exparg{B^H}{\innerproduct{B^H}{q_m}^2}
  &\to 
  \left\{
    \begin{matrix}
        0,&\qbox{$m$ is odd,}\\
        2 / (m + 1)^2,&\qbox{$m$ is even.}
    \end{matrix}
  \right.
\end{align*}
So we expect even moments to influence the score more for $0 < H \ll 1/2$, and odd for $1/2 \ll H < 1$. This is confirmed in figure~\ref{fig:fractional-brownian-1d-moments}.
\begin{figure}[h]
  \centering
  \includegraphics[width=0.9\textwidth]{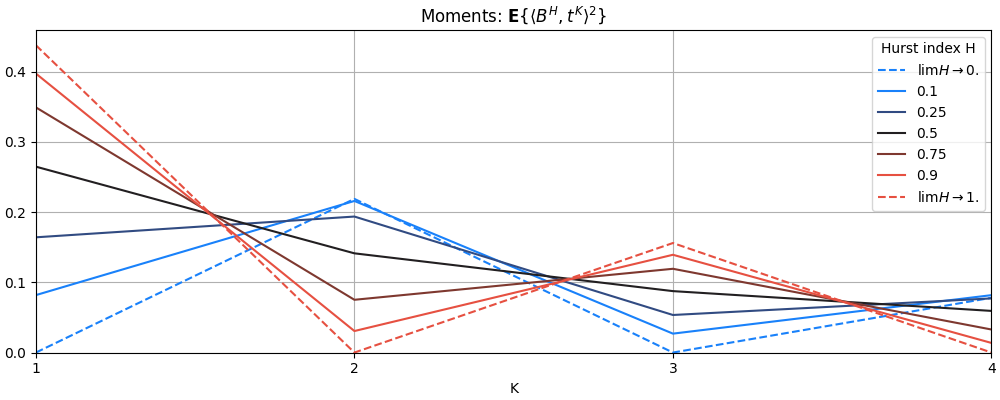}
  \caption{\textbf{Moments of fractional motion:} Shows that, for smaller $H < 1/2$, $\innerproduct{B^H}{t^k}$ is larger for even $k$. For $H > 1/2$, $\innerproduct{B^H}{t^k}$ is larger for odd $k$.}
    \label{fig:fractional-brownian-1d-moments}
\end{figure}

\subsection{Extending with Gaussian free fields}
\label{section:extending-with-GFF}
To extend Brownian distance while preserving the Cram\'er-von Mises integral form~\eqref{align:energy-distance-1d-ito-isometry}, we replace Brownian motion with Gaussian free fields (GFF). We call the result \emph{Dirichlet energy distance}.

The Gaussian free field (GFF) may be defined as the centered Gaussian field whose covariance is the Green's function of the Laplacian. For $d > 1$, the variance is unbounded, and the GFF takes values in a space of generalized functions.

When $\calX=\Rd$, this means $\Cov{G(x)G(y)}=N_d(x, y)$, where
\begin{align*}
  \begin{split}
    N_1(x, y) &= \frac{1}{2}\left[ |x| + |y| - |x - y| \right],
    \quad
    N_2(x, y) = \frac{1}{2\pi}\log\frac{1}{\|x - y\|},
    \\
    N_d(x, y) &= \frac{1}{(d-2)\Vol(\dsphere)}\frac{1}{\|x - y\|^{d-2}}.
    \qbox{($d>2$)},
  \end{split}
\end{align*}
Our choice of $N_1$ is atypical, but done so that $G$ is Brownian motion.

If $\calX=\Rd$,
let $\phi$ be the potential solving $\Delta\phi=0$, and $\phi(x)\to0$ as $\|x\|\to\infty$. Define the gradient field $Q:=\nabla\phi$.
Then so long as Green's identities hold,
\begin{align}
  \label{align:dirichlet-energy-distance}
  \D(X, Y; G)
  &= \int_\calX\int_\calX q(x)q(y) N_d(x, y)\dx\dy 
  = \int_\calX \|Q(x)\|^2\dx,
\end{align}
is the Dirichlet energy.
We can apply Fubini and obtain~\eqref{align:dirichlet-energy-distance} if $\int |q(x)x|\dx<\infty$ for $d=1$, $\int |q(x)|\log\|x\|\dx < \infty$ when $d=2$, and $q\in L^2\cap L^1$ for $d\geq3$ using a Hardy-Littlewood-Sobolev inequality~\cite{Lieb1983-vd}.
In that case, the Dirichlet energy is characteristic over $\mu-\nu$ with density $q$ satisfying the dimension dependent conditions.
In the electromagnetic analog, $Q$ is an electric field induced by charge distribution $q$.

In one dimension, $Q(x) = \int_{-\infty}^x q(t)\dt = \mu(-\infty, x) - \nu(-\infty, x)$ is the difference of cumulative distribution functions, and then \eqref{align:energy-distance-1d-ito-isometry} matches \eqref{align:dirichlet-energy-distance}.
In multiple dimensions we have replaced integration by parts in~\eqref{align:energy-distance-1d-ito-isometry} with Green's identities.

If $\calX$ is compact, we can make precise statements about sample paths.  Similar to~\cite{Werner2020-sl} section 3, we have an orthonormal basis for $L^2$, $\left\{ \psi_k \right\}$, with $(-\Delta)\psi_k = \lambda_k\psi_k$. Unlike~\cite{Werner2020-sl}, our $\psi_k$ satisfy a \emph{Neumann} boundary condition, so $\lambda_0=0$, $\psi_0$ is constant, and on the other hand when $k>0$, $\lambda_k > 0$ and $\int_\calX \psi_k\dx=0$.  Since we only integrate our GFF against a difference of measures $\mu-\nu$ satisfying $(\mu-\nu)(\calX)=0$, we can drop the $\lambda_0$ term and represent the GFF by
\begin{align*}
  G(x) &= \sum_{k=1}^\infty \frac{\xi_k}{\lambda_k^{1/2}}\psi_k(x),
\end{align*}
where $\xi_k$ are \iid scalar Gaussians.

For $s > d/2 - 1$, let $K^s$ be the completion of smooth functions with $\int_\calX q(x)\dx=0$, under the norm
\begin{align*}
  \|q\|_s^2 &= \sum_{k=1}^\infty \lambda_k^s \left( \int_\calX q(x)\psi_k(x)\dx \right)^2.
\end{align*}
For $q\in K^s$, set $q_k := \int_\calX q(x)\psi_k(x)\dx$, and then, as in~\cite{Werner2020-sl},
\begin{align*}
  \innerproduct{G}{q}
  &= \sum_{k=1}^\infty \frac{\xi_k}{\lambda_k^{1/2}}q_k
  \leq \left(\sum_{k=1}^\infty \frac{\xi_k^2}{\lambda_k^{1+s}} \right)^{1/2}
  \left( \sum_{k=1}^\infty \lambda_k^sq_k \right)^{1/2}
  \leq C(s) \|q\|_s,
\end{align*}
with finite~\as~$C(s)$ depending only on $s$ and the realization of $G$.
Therefore, $G$ is characteristic over $\mu-\nu$ with densities in $K^s$ and sample paths~\as~are in the dual $K^{-s}$.

\subsection{Summary of fractional field distances}
\label{section:summary-of-fractional-distances}
We can use centered continuous fields (section~\ref{section:extending-with-continuous-fractional-fields}) with stationary increments
\begin{align*}
  \Exp{B^H(X)B^H(Y)} &\propto \|X\|^{2H} + \|Y\|^{2H} - \|X - Y\|^{2H}
  \longleftrightarrow
  \varphi(\omega) \propto \frac{1}{\|\omega\|^{d+2H}}.
\end{align*}
These are characteristic when $\mu-\nu$ have finite $\kappa$ moments, and $H < \kappa$.
Note that the kernel representation is only finite once $2H\leq\kappa$.
Standard energy distance uses $H=1/2$.

We could have extended our Gaussian free fields using the fractional Laplacian and Riesz potentials. Then, we would have centered, stationary, Gaussian fractional free fields, which, for $d\geq3$ take the form
\begin{align*}
  \Exp{G^\alpha(X)G^\alpha(Y)} &\propto \|X - Y\|^{2\alpha - d}
  \longleftrightarrow
  \varphi(\omega) \propto \frac{1}{\|\omega\|^{2\alpha}},
\end{align*}
which are characteristic for smooth $\mu-\nu\in K^s$, for $s > d/2 - \alpha$, and $0<\alpha<d$.
Standard Dirichlet energy distance uses $\alpha=1$.

To demonstrate the utility of these different fields, consider the centered multivariate student T.  Given degrees of freedom $\dof>0$, and scale matrix $\Gamma^{1/2}$, this has density
\begin{align*}
    g(x) &\propto \left[ 1 + \frac{1}{\dof} x\cdot \Gamma^{-1}x \right]^{-(\dof+d)/2}.
\end{align*}
This distribution has finite $k^{th}$ moment only for $k<\dof$, so the continuous energy score will only be defined once $\dof>1$. On the other hand, since, $g$ and all derivatives are in $L^2(\Rd)$, these distributions will be in $K^s$, and the Dirichlet energy is always well defined.

We test this numerically in dimension $d=16$ with the traditional energy distance ($H=1/2$) and fractional GFF distance with $\alpha=-5$.
The degrees of freedom for both $\mu$, $\nu$ is $\dof\in[0.25, 0.5, 1, 2, 3]$. In all cases, $\nu$ uses $\Gamma=I_{16}$, and $\mu$ uses $\Gamma_{\tau,\sigma}$ with $(i, j)$ entry
\begin{align*}
    \Gamma_{\tau,\sigma}(i,j)
    &= \sigma(i)\sigma(j)\exp\left\{ -(|i-j|/(16\tau))^2 \right\},
\end{align*}
For correlation length $\tau = [0.01, 0.03, 0.05, 0.1]$, and random scale $\sigma$. We generate
\begin{align*}
    \sigma\sim\calU([0.7, 0.8]),
    \qbox{and}
    \sigma\sim\calU([1.2, 1.3]).
\end{align*}
We show the signal to noise ratio in figure~\ref{fig:snr-sweep-with-dirichlet}. The Dirichlet energy is well defined, and finite in 32-bit, for $\dof \geq 0.5$. It was finite even for $\dof=0.25$, but even with 4M samples the sample energy was negative about one third of the time!
Using $\alpha=1$ was too singular to sample from.
The continuous energy \texttt{NaN} in 32-bit precision until $\dof > 1$.
We emphasize that sampling from this distribution was difficult.
For example, we had to rescale the Dirichlet energy to prevent float underflow, and if we made $\mu$ closer to $\nu$, energy would be negative too often for the statistics to be reliable.

\begin{figure}[h]
  \centering
  \includegraphics[width=0.9\textwidth]{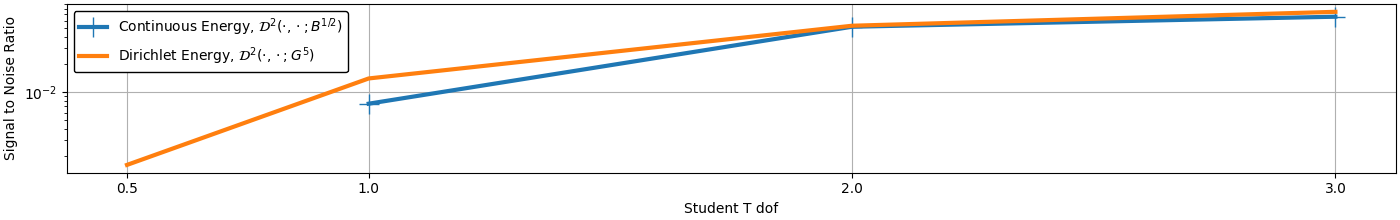}
  \caption{\textbf{Signal to noise sweep comparing continuous and Dirichlet energies:} SNR as a function of multivariate Student's T dof $\dof$. Computed with 4M samples using the kernel forms of the scores. Specfically, $H=1/2$ and $\alpha=5$ from section~\ref{section:summary-of-fractional-distances}. The Dirichlet energy is well defined even when the distribution has non-finite mean. The kernel form of continuous energy requires degrees of freedom $\dof > 2H=1$, otherwise results became \texttt{NaN} in 32-bit.}
    \label{fig:snr-sweep-with-dirichlet}
\end{figure}

\subsection{Signal to noise ratio and the Hurst index}
\label{section:SNR-and-H}
An unbiased estimate of $\D$ uses the kernel form and $N$ samples of $X$ and $Y$.
\begin{align*}
  \D_N :&=
  \frac{1}{2 N(N-1)}\sum_{\substack{n, m=1\\n\neq m}}^N\left(k(X_n,X_m) + k(Y_n,Y_m)\right)
  -\frac{1}{N}\sum_{n=1}^N k(X_n, Y_n).
\end{align*}
Alternatively, a biased estimate is
\begin{align*}
  \Dtilde_N :&=
  \frac{1}{2 N^2}\sum_{\substack{n, m=1\\n\neq m}}^N\left( k(X_n,X_m) + k(Y_n,Y_m) \right)
  -\frac{1}{N}\sum_{n=1}^N k(X_n, Y_n).
\end{align*}
The biased estimate is useful to study since
\begin{align*}
  \Dtilde_N
  &= \Exparg{U}{\innerproduct{U}{\mu_N-\nu_N}^2},
\end{align*}
for empirical measures $\mu_N$, $\nu_N$.

We want to study the relation between the random field $U$ and the signal to noise ratio when using an unbiased estimate. Define $\gamma := \mu_N - \nu_N - (\mu-\nu)$, and use $\Exp{\cdot}$ to denote expectation over both samples and fields. The signal is
\begin{align*}
  \Exparg{\gamma}{\D_N} &= \Exp{\innerproduct{U}{\mu-\nu}^2}.
\end{align*}
As for the noise, we can use $\Var{\Dtilde_N}$ as a stand-in for $\Var{\D_N}$, incurring an error that is $O(1/N)$ smaller than the leading term (which itself is $O(1/N)$). To that end, note
\begin{align*}
  \Dtilde_N
  &= \Exp{\innerproduct{U}{\mu-\nu}^2}
  + 2\Exparg{U}{\innerproduct{U}{\mu-\nu}\innerproduct{U}{\gamma}}
  + \Exparg{U}{\innerproduct{U}{\gamma}^2}, \\
  \Exparg{\gamma}{\Dtilde_N}
  &= \Exp{\innerproduct{U}{\mu-\nu}^2}
  + \Exp{\innerproduct{U}{\gamma}^2}.
\end{align*}

If the difference of measures $\mu-\nu$ is much smaller than the finite sample error $\gamma$, we ignore the $\innerproduct{U}{\mu-\nu}\innerproduct{U}{\gamma}$ term to get
\begin{align}
    \mathrm{SNR}_{\mu-\nu\ll\gamma}^2
  &\approx
  \frac{\left( \Exp{\innerproduct{U}{\mu-\nu}^2}\right)^2}{\Var{\Exparg{U}{\innerproduct{U}{\gamma}}^2}}
  \approx
  \frac{\left( \Exp{\innerproduct{U}{\mu-\nu}} \right)^2}{2\Var{\Exparg{U}{\innerproduct{U}{\nu - \nu_N}}^2}},
  \label{align:SNR2-eta-greater-than-q}
\end{align}
where the second approximation is due to $\mu - \mu_N\approx \nu - \nu_N$ (in distribution).
With the approximation~\eqref{align:SNR2-eta-greater-than-q} the noise is independent of $X$. If we consider different candidate distributions $X = X_\theta$ (depending on parameters $\theta$), the SNR will only depend on the signal $\Exp{\innerproduct{U}{\mu-\nu}^2} = \D(X_\theta, Y)$. Using this fact and our moment analysis of section~\ref{section:extending-with-continuous-fractional-fields} we can conjecture the effect of perturbations of the first and second moments on SNR. We expect SNR to favor larger $H$ when perturbing the mean, and small $H$ when perturbing the second moment.

This hypothesis is born out when we set $\nu = \calN(0, 1)$, and let $\mu$ be various perturbations. See figure~\ref{fig:snr-sweep}. For each perturbation, we computed the 32 sample unbiased distance $\D(\mu-\nu; B^H)$ for a range of $H\in(0, 1)$. This was repeated one million times, and the mean score and standard deviation were estimated. Note that the standard deviation was the same when the biased score was used (not shown). The resultant signal to noise ratio can be used to choose an optimal $H$ for each situation. When the variance is perturbed, lower $H$ is optimal, and when the mean is perturbed, higher $H$ is.
\begin{figure}[h]
  \centering
  \includegraphics[width=0.95\textwidth]{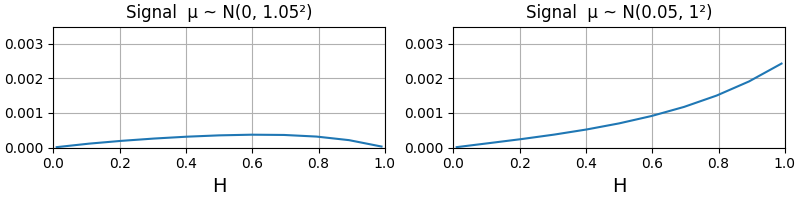}
  \includegraphics[width=0.95\textwidth]{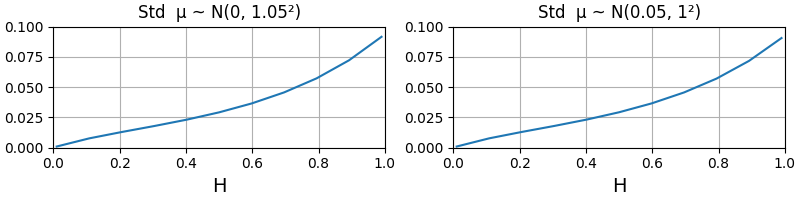}
  \includegraphics[width=0.95\textwidth]{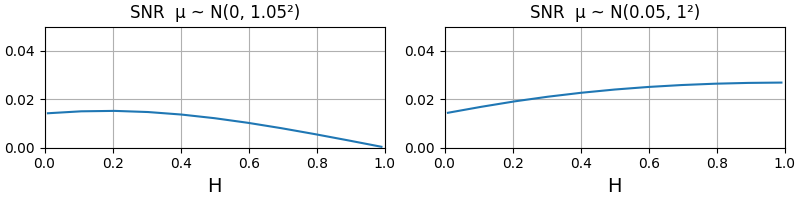}
  \caption{\textbf{SNR Sweep:} Shows that the signal, $\D(\mu-\nu; B^H)$ is the differentiating factor in SNR calculations, since the standard deviation is mostly perturbation independent. The signal shape (with respect to $H$) is as expected due to the path-wise heuristics from section~\ref{section:SNR-and-H}.}
    \label{fig:snr-sweep}
\end{figure}

\section{Other examples}
\label{section:other-examples}

\subsection{Metrics on discrete spaces}
\label{section:scores-on-discrete-spaces}
In evaluation of discrete forecasts, such as ``rainfall greater than 2mm'', it is common to use the Brier score. If the forecast probability is $p$, and the observation is a binary event $Y$, the \emph{Brier score} is $(p - Y)^2$. If $\Exp{Y=1}=q$, then the expected Brier score is
\begin{align*}
  \Exp{(p - Y)^2}
  &= (p - q)^2 + q(1-q).
\end{align*}
Up to a term independent of the forecast, this is equal to a distance $(p - q)^2$.

The inducing field on $\calX = \left\{ 0, 1 \right\}$ is
\begin{align*}
  D(x) = D_0\one_{x=0} + D_1\one_{x=1},
\end{align*}
where $D_0, D_1$ are uncorrelated with $\Exp{D_i}=0$, $\Exp{D_i^2}=1/2$. Then
\begin{align*}
  \D(p, q; D)
  &= \Exparg{D}{\Exparg{X,Y}{D(X)-D(Y)\g D}^2} \\
  &= \Exparg{D}{ \left\{ D_0( (1-p) - (1-q)) + D_1(p-q) \right\}^2 } \\
  &= (p-q)^2.
\end{align*}

This can be generalized to $\calX=\left\{ 1, 2, \ldots \right\}$ with uncorrelated sequence $\left\{ D_1,D_2,\ldots \right\}$ satisfying
\begin{align*}
  D(x) &= \sum_{k=1}^\infty D_k\one_{x=k},
  \qbox{where}
  \Exp{D_k}\equiv0,\quad \Exp{D_k^2} = 2^{-k}.
\end{align*}
Then,
\begin{align*}
  \D(\mu-\nu; D)
  &= \sum_{k=1}^\infty \frac{1}{2^k} \left( \mu(k) - \nu(k) \right)^2,
\end{align*}
is characteristic over differences of finite measures on $\left\{ 1,2,\ldots \right\}$.


\subsection{Additive Brownian motion}
\label{section:additive-brownian-motion}
Let $B_1,\ldots,B_d$ be independent one dimensional Brownian motions, and define the \emph{additive Brownian motion} as
\begin{align*}
    A(x) &= \sum_{i=1}^d B_i(x_i).
\end{align*}
Since each field $B_i$ is constant over components $j\neq i$, the inner product will factor nicely.
With $\mu_i,\nu_i$ the $i^{th}$ marginals,
\begin{align*}
    \innerproduct{A}{\mu-\nu}
    &= \sum_{i=1}^d \innerproduct{B_i}{\mu_i - \nu_i}.
\end{align*}
Using independence of the $B_i$, the distance factors into a sum of one dimensional Brownian distances.
\begin{align*}
    \D(\mu-\nu; A)
    &= \sum_{i=1}^d \Exparg{B_i}{\innerproduct{B_i}{\mu_i - \nu_i}^2}
    = \sum_{i=1}^d \D(\mu_i, \nu_i; B).
\end{align*}
The kernel form replaces the $L^2$ norm of the energy distance with the $L^1$ norm,
\begin{align*}
    \D(\mu-\nu; A)
    &= \E\|X - Y\|_1
    - \frac{1}{2}\E\|X - X'\|_1
    - \frac{1}{2}\E\|Y - Y'\|_1.
\end{align*}
This distance clearly is \emph{not} characteristic. Nonetheless, it has been useful as a loss function to learn weather and climate models~\cite{Kochkov2024-we,Lang2024-ky}

\section{Natural assumptions and requirements for sample paths}
\label{section:natural-assumptions-for-the-field}

The desires for simplicity of representation, and effective finite sample estimation of $\D$, lead to natural assumptions and requirements on the sample paths.

Since $(\mu-\nu)(\calX)=0$, the form $\innerproduct{U}{\mu-\nu}$ is unaltered if $U$ is shifted by a constant. So we may as well assume $\Exp{U}=0$, or if sample paths are continuous, we may instead assume $U(x_0)=0$, for some $x_0\in\calX$.

The variogram form~\eqref{align:distance-stochastic-integral} shows that the score is invariant to shifts $X\mapsto X + \delta$, $Y\mapsto Y + \delta$ for all $X$, $Y$ just when $U(x) - U(y)$ and $U(x+\delta)- U(y + \delta)$ have identical first and second moments for all $\delta\in\Rd$.
Therefore, a desiderata of shift invariance leads to require $U$ has wide sense stationary increments.

The second desiderata is that the finite sample signal to noise ratio be invariant under a rescaling $X\mapsto\sigma X$ and $Y\mapsto\sigma Y$.
Analysis from section~\ref{section:SNR-and-H} shows this consideration will require $U(\sigma x) \sim g(\sigma)U(x)$ for some function $g$. The consistency criteria $U(\sigma_1 \sigma_2 x) \sim g(\sigma_1)g(\sigma_2)U(x) \sim g(\sigma_1 \sigma_2)U(x)$ means $g(\sigma) = \sigma^H$ for some $H\in\Rone$. In other words, we will require $U(\sigma x) \sim \sigma^HU(x)$.

These first two desiderata significantly restrict the available Gaussian fields.
In particular, if we require the field to have stationary increments, $\lim_{x\to0} U(x)\to U(0)$ in distribution, and finite variance, then we must have $H\in[0, 1]$ (\cite{Cohen2013-zp} proposition 3.2.1).
In one dimension, the only finite-variance field satisfying this is (a multiple of) continuous fractional Brownian motion (\cite{Cohen2013-zp} corollary 3.2.1), as detailed in section~\ref{section:extending-with-continuous-fractional-fields}.
In dimension $d\geq1$, a finite-variance field satisfying this must have (\cite{Cohen2013-zp} section 3.3.1) spectral density (in polar coordinates)
\begin{align*}
    \varphi(r\theta) &\propto \frac{1}{r^{d + 2H}} S(\theta).
\end{align*}
Setting $S$ to Lebesgue measure on $\dsphere$ results in the multi-dimensional continuous fractional fields of section~\ref{section:extending-with-continuous-fractional-fields}.
Clearly though, in multiple dimensions there are other possibilities.

The third desiderata will determine what space sample paths live in.
In practice, $\mu$, $\nu$ are probability measures on $\calX$.  Their difference $\mu-\nu$ therefore takes values in some subset of finite signed Borel measures with total mass zero.
If $\calX$ is compact, then the measures are in the Banach dual $\calB^\ast(\calX)$ to the continuous functions $f\in C(\calX)$, equal to zero at some point.
To generalize this correspondence, we start by assuming $\mu-\nu\in\calB^\ast$ for some Banach space $\calB$. The pairing $\innerproduct{U}{\mu-\nu}\in\Rone$ is well defined then if $U\in \calS\subset (\calB^\ast)^\ast$. Typically the \emph{bidual} $(\calB^\ast)^\ast$ is difficult to characterize, so we would be glad to constrain $U$ to some subset. We are helped out by our third desiderata that sample score estimates converge. In other words, replacing $\mu$ with its empirical measure $\mu_n$, we desire $\innerproduct{U}{\mu-\mu_n}\to0$ for almost every $U$.
This can be rephrased as requiring $\calS$ to be a set of weak-$\ast$ continuous functionals on $\calB^\ast$. The largest space of such functionals is the original space, $\calB$ (theorem 4.20~\cite{Reed1980-gk}). Our sample estimate convergence desiderata has therefore led us to require $U$ takes values in a Banach space $\calB$, and the difference of measures $\mu-\nu\in\calB^\ast$.

A related assumption is that $\E\|U\|_\calB^2<\infty$, since then
\begin{align*}
  \D(\mu_n, \nu) &= \Exparg{U}{\innerproduct{U}{\mu_n-\nu}^2}
  \leq \E\|U\|_\calB^2 \|\mu_n-\nu\|_{\calB^\ast}^2,
\end{align*}
is well defined. In particular, if $\mu_n\to\nu$ strongly, then $\D(\mu_n,\nu)\to0$.
See~\cite{Vayer2023-bx,Modeste2024-xa} for connections between $\D$ and Wasserstein distances.

\section{Support theorems for characteristic fields}
\label{section:forgery-theorems}
Given Banach space $\calB$, we say $f\in\calB$ is in the \emph{support} of $U$ if, for every $\eps>0$, $\rmP[\|U - f\|_\calB < \eps] > 0$. Let $\support$ be the support of $U$ in $\calB$, and $\supportspan$ its linear span.
\begin{theorem}
  \label{theorem:characteristic-implies-dense}
  Suppose sample paths of $U$ take values in $\calB$ and $\E\|U\|_\calB^2<\infty$. Then $U$ is characteristic over $\calB^\ast$ if and only if $\supportspan$ is dense in $\calB$.
\end{theorem}

Since the support of a Gaussian process is a closed linear subspace, \cite{Vakhania1975-ms}, we have
\begin{corollary}
  Gaussian field $U$ with sample paths~\as~in $\calB$ is characteristic over $\calB^\ast$ if and only if $\support = \calB$. In that case, for every $f\in\calB$, $\eps>0$,
    $\rmP\left[ \|U - f\|_\calB < \eps \right] > 0$.
  \label{corollary:characteristic-gaussian-implies-forgery}
\end{corollary}

Support theorems of this type are not surprising, since $U$ is characteristic just when, for all $\eta$, $\innerproduct{U}{\eta} \neq0$ on a set of positive measure. So $U$ must be ``similar enough'' to arbitrary elements of $\calB^\ast$. For non-Gaussian fields we cannot expect the support to be dense. For example, suppose $\calX$ has basis $(\psi_1, \psi_2,\ldots)$, and let $\xi = (\xi_1,\xi_2,\ldots)$ be a sequence of mutually exclusive random variables, so that $\xi_k=1$ (and all others $=0$) with probability $2^{-k}$. Then $U = \sum_{k=1}^\infty \xi_k \psi_k$ is characteristic but paths are clearly not dense. The paths of $U$ are \emph{equal} to a basis, which is why our result is close to (but not the same as) ``the support is dense in a basis.''

\subsection{Examples of support theorems}
\label{section:forgery-examples}
We give some examples of support theorems here.
Support theorems for Brownian motion are common enough to be standard exercises (\cite{Morters2010-vx}, 1.8). Since the support of Gaussian field $U\in\calB$ is the closure of the reproducing kernel Hilbert space, a variety of support theorems could come out of that viewpoint on distances. See~\cite{van-der-Vaart2008-wv} lemma 5.1 and~\cite{Gretton2008-kv} theorem 3.

\subsubsection{Continuous paths and distributions on compact spaces}

Let $\calX$ be compact. Choosing $x_0\in\calX$, designate a subspace of continuous functions
\begin{align*}
  C_{x_0}(\calX) &= \left\{ f\in C(\calX)\st f(x_0)=0 \right\}.
\end{align*}

Riesz-Markov and the dual relation for quotients (\cite{Reed1980-gk} theorem 4.14, and problem 3.30) give
\begin{align*}
    (C_{x_0}(\calX))^\ast &= \F(\calX) = \left\{ \mbox{finite signed Borel measures $\eta$ on $\calX$} \st \eta\left( \calX \right)=0 \right\}.
\end{align*}
Corollary~\ref{corollary:characteristic-gaussian-implies-forgery} then says, a Gaussian field with paths~\as~in $C_{x_0}(\calX)$ is characteristic over $\F(\calX)$ just when it its support is dense.

As an application, since Brownian paths on $[0, T]$ are characteristic over $\F([0, T])$ (corollary~\ref{corollary:continuous-fractional-fields-are-characteristic}), their support in $C_{0}([0, T])$ is equal to $C_0([0, T])$.

\subsubsection{Continuous paths of sub-polynomial growth and distributions with finite moments}
For $\alpha > 0$, set
\begin{align*}
  C_{0,\alpha} :&= \left\{ U \in C(\Rd) \st U(0) = 0,\qbox{and}\lim_{\|x\|\to\infty}\frac{\|U(x)\|}{1 + \|x\|^\alpha} = 0 \right\}.
\end{align*}
The fractional fields of section~\ref{section:extending-with-continuous-fractional-fields} have variance $\E B^H(x)^2=\|x\|^{2H}$, and so $B^H(x) / (1 + \|x\|^{H + \eps})\to0$~\as~for any $\eps>0$ and thus $B^H\in C_{0,\alpha}$ if $H<\alpha$. See \cite{Orey1972-am} or \cite{Cohen2013-zp} theorem 3.2.4.

To find the dual, define $(Mf)(x) = f(x)(1 + \|x\|^\alpha)^{-1}$, and set
\begin{align*}
  A :&= \left\{ f \st f\in C(\Rd),\, (Mf)(x)\to0 \mbox{ as } \|x\|\to\infty \right\},
\end{align*}
If we push $A$ forward by $M$, we get $C_0(\Rd)$, the continuous functions vanishing at infinity.
The dual of $C_0(\Rd)$ is $\F(\Rd)$ (\cite{Rudin1987-yf} theorem 6.19).
Therefore $A^\ast$ is the signed Borel measures $\eta$ with $\int (1 + \|x\|)^\alpha\deta<\infty$.
Finally, $C_{0,\alpha}$ is isomorphic to the quotient space $(A / \left\{ \mbox{constant functions} \right\})$.
So, using the dual relation for quotients (\cite{Reed1980-gk} problem 3.30)
\begin{align*}
    (C_{0,\alpha})^\ast &= \left\{ \eta\in\F(\Rd) \st \int (1 + \|x\|^\alpha)\dabseta(x) < \infty,\, \int \deta(x)=0 \right\}.
\end{align*}
Since our continuous Gaussian fractional fields $B^H$, are characteristic over $(C_{0,\alpha})^\ast$ once $H < \alpha$ (corollary~\ref{corollary:continuous-fractional-fields-are-characteristic}), their support in $C_{0,\alpha}$ is $C_{0,\alpha}$.

\subsubsection{Gaussian free fields and smooth densities}
We will follow section~\ref{section:extending-with-GFF}, and identify members of Sobolev spaces $q\in K^s$ with their induced measure. Suppose $\calX\subset\Rd$ is compact, and $s > d/2 - 1$. Then the GFF is in $K^{-s}$, and is characteristic over $K^s = (K^{-s})^\ast$. It's support in $K^{-s}$ is therefore $K^{-s}$.

\subsection{Proofs of support theorems}
\label{section:forgery-proofs}

Given Banach space $\calB$ and subset $A\subset \calB$, the \emph{annihilator}
\begin{align*}
  A^\perp :&= \left\{ v\in \calB^\ast\st \innerproduct{a}{v}=0,\mbox{ for every }a\in A \right\}.
\end{align*}
\begin{lemma}
  Suppose $\calB$ is a Banach space, and $A\subset \calB$ is a closed subspace.
  Then $A^\perp=\left\{ 0 \right\}$ implies $A=\calB$.
  \label{lemma:trivial-annihilator-means-its-the-whole-space}
\end{lemma}
\begin{proof}
  If $A\subsetneq \calB$, then by Hahn Banach (see~\cite{Reed1980-gk}, chapter 3, corollary 3), there exists nonzero $v\in \calB^\ast$, $\innerproduct{a}{v}=0$, for every $a\in A$, whence $A^\perp\neq \left\{ 0 \right\}$.
\end{proof}

\begin{lemma}
  Subsets $A\subset\calB$ with positive measure contain support points. That is,
  $\rmP\left[ U\in A \right] > 0$ implies $A \cap \support\neq\emptyset$
  \label{lemma:positive-measure-subsets-contain-support-points}
\end{lemma}
\begin{proof}
  Suppose $A\cap\support=\emptyset$, then for every $a\in A$, there exists open $\calO_a\subset \calB$ with $\rmP\left[ U\in\calO_a \right]=0$. Then,
  \begin{align*}
    \rmP\left[ U\in A \right] &\leq \rmP\left[ U\in\bigcup_{a\in A}\calO_a \right]
    \leq \sum_{a\in A}\rmP\left[ U\in\calO_a \right] = 0.
  \end{align*}
\end{proof}

We can now prove both directions of~\eqref{theorem:characteristic-implies-dense}.
\begin{proof}[Proof of theorem~\eqref{theorem:characteristic-implies-dense}]
  Suppose $U$ is characteristic over $\calB^\ast$. Then for every nonzero $\eta\in\calB^\ast$, $\rmP\left[ \innerproduct{U}{\eta} \neq 0\right] > 0$. By lemma~\ref{lemma:positive-measure-subsets-contain-support-points} $\left\{ \innerproduct{U}{\eta} \neq 0 \right\}$ contains an support point. Therefore
  \begin{align*}
    \left\{ 0 \right\}
    &=\left\{ \eta\in\calB^\ast\st \innerproduct{v}{\eta}=0 \mbox{ for every } v\in\support \right\} \\
    &\supset
    \left\{ \eta\in\calB^\ast\st \innerproduct{v}{\eta}=0 \mbox{ for every } v\in\supportspanclosure \right\},
  \end{align*}
  where $\supportspanclosure$ is the closure in $\calB$ of $\supportspan$.
  Therefore, the annihilator $\supportspanclosure^\perp=\left\{ 0 \right\}$, and then by lemma~\ref{lemma:trivial-annihilator-means-its-the-whole-space}, $\supportspanclosure=\calB$.

Conversely, suppose $\supportspan$ is dense in $\calB$.
  Let nonzero $\eta\in\calB^\ast$, and choose $v\in\calB$ with $\innerproduct{v}{\eta} > 0$.
  Since $\supportspan$ is dense, there exists support points $\left\{ v_1,\ldots,v_N \right\}$ so that
  \begin{align*}
    \left\|v - \sum_{n=1}^N \beta_n v_n\right\|
    & < \frac{|\innerproduct{v}{\eta}|}{\|\eta\|}.
  \end{align*}
  Therefore,
  \begin{align*}
    \sum_{n=1}^N \beta_n\innerproduct{v_n}{\eta}
    &= \innerproduct{v}{\eta} - \innerproduct{v - \sum_{n=1}^N\beta_nv_n }{\eta}
    &\geq \innerproduct{v}{\eta} - \left\|v - \sum_{n=1}^N\beta_nv_n\right\| \|\eta\|
    > 0,
  \end{align*}
  which implies $|\innerproduct{v_j}{\eta}| > 0$ for some $v_j$.

  Define the open set of paths
  \begin{align*}
    \calO &= \left\{ \|U - v_j\| < \frac{|\innerproduct{v_j}{\eta}|}{\|\eta\|} \right\}.
  \end{align*}
  It follows that for $U\in\calO$,
  \begin{align*}
    \left|\innerproduct{U}{\eta}\right| 
    &= |\innerproduct{v_j}{\eta} - \innerproduct{U - v_j}{\eta}|
    \geq |\innerproduct{v_j}{\eta}| - \|U - v_j\|\|\eta\|
    > 0.
  \end{align*}
  Moreover, since $v_j$ is a support point, $\rmP\left[ \calO \right] = \delta > 0$.
  We can therefore conclude
  \begin{align*}
    \rmP\left[ |\innerproduct{U}{\eta}| > 0 \right]
    &\geq \rmP\left[ |\innerproduct{U}{\eta}| > 0 \g U\in\calO\right] \rmP\left[ U\in\calO \right]
    = 1\,\delta
    = \delta
    > 0,
  \end{align*}
  and $U$ is characteristic.
\end{proof}

\section*{Acknowledgments} The author would like to acknowledge Srinivas Vasudevan for many conversations about stochastic processes in general, as well as some specific to this work.

\bibliographystyle{plain}                                                                          
\bibliography{kernel-distances-stochastic-integrals}

\begin{thebibliography}{10}

\bibitem{Cohen2013-zp}
Serge Cohen and Jacques Istas.
\newblock {\em {Fractional Fields and Applications}}.
\newblock Mathematiques \& Applications (Berlin). Springer, 2013 edition,
  30~May 2013.

\bibitem{Durrett2019-fw}
Rick Durrett.
\newblock {\em {Probability: Theory and Examples}}.
\newblock Cambridge University Press, April 2019.

\bibitem{Gneiting2007-ms}
Tilmann Gneiting and Adrian~E Raftery.
\newblock {Strictly Proper Scoring Rules, Prediction, and Estimation}.
\newblock {\em J. Am. Stat. Assoc.}, 102(477):359--378, 1~March 2007.

\bibitem{Gretton2008-kv}
Arthur Gretton, Karsten Borgwardt, Malte~J Rasch, Bernhard Scholkopf, and
  Alexander~J Smola.
\newblock {A kernel method for the two-sample problem}.
\newblock {\em arXiv [cs.LG]}, 15~May 2008.

\bibitem{Gretton2012-fp}
Arthur Gretton, Karsten~M Borgwardt, Tuebingen~Mpg De, Malte~J Rasch,
  Xinjiekouwai St, Bernhard Sch{\"{o}}lkopf, and Alexander Smola.
\newblock {A Kernel Two-Sample Test}.
\newblock {\em J. Mach. Learn. Res.}, 13(25):723--773, 2012.

\bibitem{Kochkov2024-we}
Dmitrii Kochkov, Janni Yuval, Ian Langmore, Peter Norgaard, Jamie Smith,
  Griffin Mooers, Milan Kl{\"{o}}wer, James Lottes, Stephan Rasp, Peter
  D{\"{u}}ben, Sam Hatfield, Peter Battaglia, Alvaro Sanchez-Gonzalez, Matthew
  Willson, Michael~P Brenner, and Stephan Hoyer.
\newblock {Neural general circulation models for weather and climate}.
\newblock {\em Nature}, pages 1--7, 22~July 2024.

\bibitem{Lang2024-ky}
Simon Lang, Mihai Alexe, Mariana C~A Clare, Christopher Roberts, Rilwan
  Adewoyin, Zied~Ben Bouall\`{e}gue, Matthew Chantry, Jesper Dramsch, Peter~D
  Dueben, Sara Hahner, Pedro Maciel, Ana Prieto-Nemesio, Cathal O'Brien,
  Florian Pinault, Jan Polster, Baudouin Raoult, Steffen Tietsche, and Martin
  Leutbecher.
\newblock {AIFS-CRPS: Ensemble forecasting using a model trained with a loss
  function based on the Continuous Ranked Probability Score}.
\newblock {\em arXiv [physics.ao-ph]}, 20~December 2024.

\bibitem{Lieb1983-vd}
Elliott~H Lieb.
\newblock {Sharp constants in the hardy-Littlewood-Sobolev and related
  inequalities}.
\newblock {\em Ann. Math.}, 118(2):349, September 1983.

\bibitem{Modeste2024-xa}
Thibault Modeste and Cl\'{e}ment Dombry.
\newblock {Characterization of translation invariant MMD on R d and connections
  with Wasserstein distances}.
\newblock {\em J. Mach. Learn. Res.}, 25(237), May 2024.

\bibitem{Morters2010-vx}
Peter Morters and Yuval Peres.
\newblock {\em {Cambridge series in statistical and probabilistic mathematics:
  Brownian motion series number 30}}.
\newblock Cambridge University Press, 25~March 2010.

\bibitem{Orey1972-am}
Steven Orey.
\newblock {Growth rate of certain Gaussian processes}.
\newblock In {\em {Proceedings of the Sixth Berkeley Symposium on Mathematical
  Statistics and Probability, Volume 2: Probability Theory}}, volume 6.2, pages
  443--452. University of California Press, 1~January 1972.

\bibitem{Reed1980-gk}
M~Reed and B~Simon.
\newblock {\em {Methods of Modern Mathematical Physics. I {F}unctional
  Analysis}}.
\newblock Academic Press, second edition, 1980.

\bibitem{Rudin1987-yf}
Walter Rudin.
\newblock {\em {Real and complex analysis, 3rd ed}}.
\newblock McGraw-Hill, Inc., 1987.

\bibitem{Ryvkina2013-ax}
Jelena Ryvkina.
\newblock {Fractional Brownian motion with variable Hurst parameter: Definition
  and properties}.
\newblock {\em arXiv [math.PR]}, 12~June 2013.

\bibitem{Shevchenko2015-wg}
Georgiy Shevchenko.
\newblock {Fractional Brownian motion in a nutshell}.
\newblock {\em Int. J. Mod. Phys. Conf. Ser.}, 36:1560002, January 2015.

\bibitem{Sriperumbudur2009-ia}
Bharath~K Sriperumbudur, A~Gretton, K~Fukumizu, B~Scholkopf, and Gert R~G
  Lanckriet.
\newblock {Hilbert space embeddings and metrics on probability measures}.
\newblock {\em J. Mach. Learn. Res.}, 11:1517--1561, 30~July 2009.

\bibitem{Szekely2009-ho}
G\'{a}bor~J Sz\'{e}kely and Maria~L Rizzo.
\newblock {Brownian distance covariance}.
\newblock {\em Ann. Appl. Stat.}, 3(4):1236--1265, December 2009.

\bibitem{Szekely2013-qh}
G\'{a}bor~J Sz\'{e}kely and Maria~L Rizzo.
\newblock {Energy statistics: A class of statistics based on distances}.
\newblock {\em J. Stat. Plan. Inference}, 143(8):1249--1272, 1~August 2013.

\bibitem{Vakhania1975-ms}
N~Vakhania.
\newblock {The topological support of Gaussian measure in Banach space}.
\newblock {\em Nagoya Math. J.}, 57(none):59--63, 1~June 1975.

\bibitem{van-der-Vaart2008-wv}
A~W van~der Vaart and J~H van Zanten.
\newblock {Reproducing kernel Hilbert spaces of Gaussian priors}.
\newblock In {\em {Institute of Mathematical Statistics Collections}},
  volume~3, pages 200--222. Institute of Mathematical Statistics, 1~January
  2008.

\bibitem{Vayer2023-bx}
Titouan Vayer and R\'{e}mi Gribonval.
\newblock {Controlling Wasserstein Distances by Kernel Norms with Application
  to Compressive Statistical Learning}.
\newblock {\em J. Mach. Learn. Res.}, 24(149):1--51, 2023.

\bibitem{Werner2020-sl}
Wendelin Werner and Ellen Powell.
\newblock {Lecture notes on the Gaussian Free Field}.
\newblock {\em arXiv [math.PR]}, 9~April 2020.

\end{thebibliography}

\end{document}